\documentclass[oneside,british]{amsart}
\usepackage[T1]{fontenc}
\usepackage[latin9]{inputenc}
\usepackage{amsthm}
\usepackage{amssymb}
\usepackage{setspace}
\usepackage{esint}
\makeatletter
\numberwithin{equation}{section}
\numberwithin{figure}{section}
\theoremstyle{plain}
\newtheorem{thm}{\protect\theoremname}[section]
  \theoremstyle{plain}
  \newtheorem{cor}[thm]{\protect\corollaryname}
  \theoremstyle{definition}
  \newtheorem{defn}[thm]{\protect\definitionname}
  \theoremstyle{plain}
  \newtheorem{lem}[thm]{\protect\lemmaname}
  \theoremstyle{plain}
  \newtheorem{prop}[thm]{\protect\propositionname}

\makeatother

\usepackage{babel}
  \providecommand{\corollaryname}{Corollary}
  \providecommand{\definitionname}{Definition}
  \providecommand{\lemmaname}{Lemma}
  \providecommand{\propositionname}{Proposition}
\providecommand{\theoremname}{Theorem}

\begin{document}

\title{A self-similar measure with dense rotations, singular projections
and discrete slices}

\author{Ariel Rapaport}

\date{April 24, 2017}



\thanks{Supported by ERC grant 306494}
\begin{abstract}
We construct a planar homogeneous self-similar measure, with strong
separation, dense rotations and dimension greater than $1$, such
that there exist lines for which dimension conservation does not hold
and the projection of the measure is singular. In fact, the set of
such directions is residual and the typical slices of the measure,
perpendicular to these directions, are discrete.
\end{abstract}

\maketitle

\section{Introduction and statement of results}

Let $R$ be a $2\times2$ rotation matrix, with $R^{n}\ne Id$ for
all $n\ge1$, and let $r\in(0,1)$. Consider a homogeneous IFS on
$\mathbb{R}^{2}$
\[
\{\varphi_{i}(x)=rRx+a_{i}\}_{i\in I},
\]
with the strong separation condition (SSC), and a self-similar measure
\[
\mu=\sum_{i\in I}p_{i}\cdot\varphi_{i}\mu\:.
\]
It is among the most basic planar self-similar measures. Hence it
is a natural question in fractal geometry to study the dimension and
continuity of the projections $\{P_{u}\mu\}_{u\in S}$ and slices
\[
\{\{\mu_{u,x}\}_{x\in\mathbb{R}^{2}}\::\:u\in S\}\:.
\]
Here $S$ is the unit circle of $\mathbb{R}^{2}$, $P_{u}$ is the
orthogonal projection onto the line spanned by $u$, and $\{\mu_{u,x}\}_{x\in\mathbb{R}^{2}}$
is the disintegration of $\mu$ with respect to $P_{u}^{-1}(\mathcal{B})$,
where $\mathcal{B}$ is the Borel $\sigma$-algebra of $\mathbb{R}^{2}$.
A more elaborate description of these disintegrations is given in
Section \ref{S2}.

Dimensionwise, the behaviour of the projections is as regular as possible.
Indeed, Hochman and Shmerkin \cite{HS} have proven that $P_{u}\mu$
is exact dimensional, with
\[
\dim P_{u}\mu=\min\{1,\dim\mu\},
\]
for each $u\in S$. A version of this, for self-similar sets with
dense rotations, was first proven by Peres and Shmerkin \cite{PS}.
Considering the continuity of the projections, Shmerkin and Solomyak
\cite{SS} have shown, assuming $\dim\mu>1$, that the set
\[
E=\{u\in S\::\:P_{u}\mu\mbox{ is singular}\}
\]
has zero Hausdorff dimension.

Let us turn to discuss the concept of dimension conservation and the
dimension of slices. A Borel probability measure $\nu$ on $\mathbb{R}^{2}$
is said to be dimension conserving (DC), with respect to the projection
$P_{u}$, if
\[
\dim_{H}\nu=\dim_{H}P_{u}\nu+\dim_{H}\nu_{u,x}\mbox{ for \ensuremath{\nu}-a.e. \ensuremath{x\in\mathbb{R}^{2}}},
\]
where $\dim_{H}$ stands for Hausdorff dimension. It always holds
that $\nu$ is DC with respect to $P_{u}$ for almost every $u\in S$.
This follows from results, valid for general measures, regarding the
typical dimension of projections (see \cite{HK}) and slices (see
\cite{JM}). Falconer and Jin \cite{FJ1} have shown that $\nu$ is
DC, with respect to $P_{u}$ for all $u\in S$, whenever $\nu$ is
self-similar with a finite rotation group. An analogues statement,
for self-similar sets with the SSC, was first proven by Furstenberg
\cite{F}. Another related result for sets is due to Falconer and
Jin \cite{FJ2}. They showed that if $K\subset\mathbb{R}^{2}$ is
self-similar, with $\dim K>1$ and a dense rotation group, then for
every $\epsilon>0$ there exists $N_{\epsilon}\subset S$, with $\dim_{H}N_{\epsilon}=0$,
such that for $u\in S\setminus N_{\epsilon}$ the set
\[
\{x\in\mathrm{span}\{u\}\::\:\dim_{H}(K\cap P_{u}^{-1}\{x\})>\dim K-1-\epsilon\}
\]
has positive length.

Taking these results into account, it is natural to ask whether the
sets $E$, defined above, and
\[
F=\{u\in S\::\:\mu\mbox{ is not DC with respect to \ensuremath{P_{u}}}\},
\]
must be empty whenever the dimension of $\mu$ exceeds $1$. A version,
for self-similar sets, of this folklore question regarding $E$ is
asked in Section 4 of \cite{BFVZ}. The purpose of this paper is to
show that $E$ and $F$ are not necessarily empty, and in fact can
both be topologically large. The following theorem is our main result.
Recall that a measure $\nu$ is said to be discrete if it is supported
on a countable set.
\begin{thm}
\label{T1.1}There exist $r\in(0,1)$, a $2\times2$ rotation matrix
$R$ with $R^{n}\ne Id$ for all $n\ge1$, and a homogeneous planar
self-similar IFS
\[
\{\varphi_{i}(x)=rRx+a_{i}\}_{i\in I}
\]
with the SSC, such that the self-similar measure $\mu=\sum_{i\in I}|I|^{-1}\cdot\varphi_{i}\mu$
satisfies $\dim\mu>1$, and each of the sets
\begin{equation}
\{u\in S\::\:P_{u}\mu\mbox{ is singular}\}\label{E14}
\end{equation}
and
\begin{equation}
\{u\in S\::\:\mu_{u,x}\mbox{ is discrete for \ensuremath{\mu}-a.e. \ensuremath{x\in\mathbb{R}^{2}}}\}\label{E15}
\end{equation}
contains a dense $G_{\delta}$ subset of $S$.
\end{thm}
If $\nu$ is a discrete measure on $\mathbb{R}^{2}$ then clearly
$\dim\nu=0$, hence we get the following corollary.
\begin{cor}
Let $\mu$ be the self-similar measure from Theorem \ref{T1.1}, then
the set
\[
\{u\in S\::\:\mu\mbox{ is not DC with respect to \ensuremath{P_{u}}}\}
\]
contains a dense $G_{\delta}$ subset of $S$.
\end{cor}

Theorem \ref{T1.1} is related to an example obtained by Nazarov,
Peres and Shmerkin \cite{NPS}. They have presented a planar self-affine
measure $\nu$, with the SSC and having dimension greater than $1$,
such that the set of $u\in S$ for which $P_{u}\nu$ is singular contains
a $G_{\delta}$ subset. We do not pursue this, but our argument can
probably be used for showing that, for a residual set of directions
$u\in S$, the slices $\{\nu_{u,x}\}_{x\in\mathbb{R}^{2}}$ are $\nu$-typically
discrete. Also related to Theorem \ref{T1.1} is a recent paper, by
Simon and Vágó \cite{SV}, in which certain one-parameter families
of self-similar measures $\nu_{\alpha}$ on the line are constructed.
For these families it holds that the similarity dimension of $\nu_{\alpha}$
is greater than $1$ for every $\alpha$, but the set of parameters
for which $\nu_{\alpha}$ is singular is topologically large.

In our construction of $\mu$ the rotational part $rR$, of the maps
in the IFS, comes from a reciprocal of a complex Pisot number (see
definition \ref{D2.1} below). While dealing with parametric families
of measures, Pisot numbers have been used before, in several situations,
in order to demonstrate the existence of exceptional parameters for
which the corresponding measures are singular. This was first done
by Erd\H{o}s \cite{E}, who proved that the Bernoulli convolution
corresponding to $\lambda$, i.e. the distribution of the random sum
$\sum_{n}\pm\lambda^{n}$, is singular whenever $\lambda^{-1}\in(1,2)$
is Pisot. The example from \cite{NPS}, mentioned above, also utilizes
real Pisot numbers. In \cite{SX}, complex Pisot numbers are used
in order to obtain examples of singular complex Bernoulli convolutions.

As a by-product of our construction, we obtain information on the
Hausdorff measure of typical slices of self-similar sets at the critical
dimension. Let $K$ be a planar self-similar set with the SSC, and
denote by $\mathbf{m}$ the Haar measure of $S$. Write $s$ for the
Hausdorff dimension of $K$, and assume $s>1$. For $t\ge0$ denote
by $\mathcal{H}^{t}$ the $t$-dimensional Hausdorff measure. Given
$u\in S$ and $x\in K$ write $K_{u,x}$ for the slice $K\cap(x+\mathrm{span}\{u\})$.
Since $0<\mathcal{H}^{s}(K)<\infty$, the Hausdorff dimension of $K_{u,x}$
is equal to $s-1$, with finite $\mathcal{H}^{s-1}$-measure, for
$\mathcal{H}^{s}\times\mathbf{m}$-a.e. $(x,u)\in K\times S$ (see
Theorem 10.11 in \cite{M}). However, it was not known whether the
set
\[
Q=\{(x,u)\in K\times S\::\:\mathcal{H}^{s-1}(K_{u,x})>0\}
\]
must have positive $\mathcal{H}^{s}\times\mathbf{m}$-measure. In
Corollary 2.3 from \cite{R} the author has shown that if this holds,
and $K$ has dense rotations, then $P_{u}(\mathcal{H}^{s}|_{K})$
is absolutely continuous for all $u\in S$. In our example from Theorem
\ref{T1.1} we shall have $\mu=C\cdot\mathcal{H}^{s}|_{K}$, where
$C>0$ is a normalizing constant. Hence we obtain the following corollary.
\begin{cor}
\label{C1.2}Let $\{\varphi_{i}\}_{i\in I}$ be the IFS constructed
in Theorem \ref{T1.1}. Denote by $K$ its attractor, and write $s$
for the Hausdorff dimension of $K$. Then $s>1$ and, 
\[
\mathcal{H}^{s}\times\mathbf{m}\{(x,u)\in K\times S\::\:\mathcal{H}^{s-1}(K_{u,x})>0\}=0\:.
\]

\end{cor}

It is interesting to note that, in contrast with Corollary \ref{C1.2},
if $K\subset\mathbb{R}^{n}$ is self-similar, with the SSC, finite
rotation group and dimension $s$ greater than $2m$, then $\mathcal{H}^{s-m}(K\cap V)>0$
for typical affine $(n-m)$-planes $V\subset\mathbb{R}^{n}$ (see
Corollary 2.2 in \cite{R}).

The rest of the paper is organized as follows. In Section \ref{S2}
the measure $\mu$ from Theorem \ref{T1.1} is constructed. In Section
\ref{S3} we show that the set defined in (\ref{E14}) is residual.
In Section \ref{S4} we complete the proof by establishing this for
the set appearing in (\ref{E15}).$\newline$

\textbf{Acknowledgment. }This paper is a part of the author's PhD
thesis conducted at the Hebrew University of Jerusalem. I would like
to thank my advisor, Michael Hochman, for his support and useful suggestions.
I would also like to thank Or Landesberg for helpful discussions.

\section{\label{S2}Construction of $\mu$}

In this section we carry out the construction of the measure $\mu$
from Theorem \ref{T1.1}. It will be convenient to identify  $\mathbb{R}^{2}$
with the complex plane $\mathbb{C}$. We shall use some simple facts
from the theory of field extensions, for which we refer to chapters
17 and 18 in \cite{I}. Our example involves complex Pisot numbers,
which we now define.
\begin{defn}
\label{D2.1}An algebraic integer $\theta\in\mathbb{C}$ is called
a complex Pisot number if $\theta\notin\mathbb{R}$, $|\theta|>1$,
and all of the Galois conjugates of $\theta$ (i.e. the other roots
of the minimal polynomial of $\theta$), except $\overline{\theta}$,
are less than one in modulus.
\end{defn}

Given algebraic numbers $\alpha_{1},...,\alpha_{n}\in\mathbb{C}$,
we denote by $\mathbb{Q}[\alpha_{1},...,\alpha_{n}]$ the smallest
subfield of $\mathbb{C}$ containing $\alpha_{1},...,\alpha_{n}$.
If $F\subset E$ are subfields of $\mathbb{C}$, we write $[E:F]$
for the degree of the field extension $E/F$. The next lemma is probably
known, but we could not find an appropriate reference. Hence the proof,
which uses a bit of Galois theory, is given at the end of this section.
\begin{lem}
\label{L2}Let $\theta$ be a complex Pisot number with $[\mathbb{Q}[\theta]:\mathbb{Q}]=3$,
then $\arg\theta\notin\pi\mathbb{Q}$.
\end{lem}

Now let $\theta$ be a complex Pisot number such that,
\begin{itemize}
\item $\arg\theta\notin\pi\mathbb{Q}$;
\item $|\theta|$ lies in $(3,4)$;
\item the minimal polynomial of $\theta$ has constant term $1$ or $-1$.
\end{itemize}
For example, the polynomial $f(X)=X^{3}+X^{2}+10X+1$ has three roots,
\[
\theta_{0}\thickapprox-0.45+3.11i,\:\:\overline{\theta_{0}}\thickapprox-0.45-3.11i\:\mbox{ and }\:\alpha\thickapprox-0.1\:.
\]
Since $f$ doesn't have a root in $\mathbb{Z}$, it follows from Gauss's
lemma that $f$ is irreducible over $\mathbb{Q}$. Hence $f$ is the
minimal polynomial of $\theta_{0}$ over $\mathbb{Q}$, and 
\[
[\mathbb{Q}[\theta_{0}]:\mathbb{Q}]=\deg f=3\:.
\]
This shows that $\theta_{0}$ is a complex Pisot number, and from
Lemma \ref{L2} we get that $\arg\theta_{0}\notin\pi\mathbb{Q}$.
Since $3<|\theta_{0}|<4$ and the constant term of $f$ is $1$, the
number $\theta_{0}$ satisfies all of the required properties.

Write $\lambda=\theta^{-1}$ and note that $\lambda$ may be thought
of as a $2\times2$ matrix $rR$, where $r=|\lambda|$ and $R$ is
a planar rotation by angle $\arg\lambda$. From $\arg\lambda\notin\pi\mathbb{Q}$
it follows $R^{n}\ne Id$ for all $n\ge1$. Let $\mathcal{V}$ be
the set of all $(a_{1},a_{2})\in\mathbb{C}^{2}$ for which the IFS
\[
\{z\rightarrow\lambda\cdot z+(-1)^{k}a_{j}\::\:k,j\in\{1,2\}\}
\]
satisfies the strong separation condition (SSC). Since $|\lambda|<\frac{1}{3}$,
it is not hard to see that $(\frac{2}{3},\frac{2i}{3})\in\mathcal{V}$
and in particular that $\mathcal{V}\ne\emptyset$. The next lemma
is proven at the end of this section.
\begin{lem}
\label{L3}The set $\mathcal{Y}:=\{k\cdot\lambda^{l}\::\:k,l\in\mathbb{N}\}$
is dense in $\mathbb{C}$.
\end{lem}

Clearly $\mathcal{V}$ is open in $\mathbb{C}^{2}$, hence from $\mathcal{V}\ne\emptyset$
and Lemma \ref{L3} it follows that there exists
\[
(a_{1},a_{2})\in\mathcal{V}\cap(\mathcal{Y}\times\mathcal{Y})\:.
\]
For $k,j\in\{1,2\}$ and $z\in\mathbb{C}$ set
\[
\varphi_{k,j}(z)=\lambda\cdot z+(-1)^{k}\cdot a_{j},
\]
then the IFS
\[
\Phi:=\{\varphi_{k,j}\::\:k,j\in\{1,2\}\}
\]
satisfies the SSC. Denote by $\mathcal{M}(\mathbb{C})$ the collection
of all compactly supported Borel probability measures on $\mathbb{C}$.
Let $\mu$ be the unique member of $\mathcal{M}(\mathbb{C})$ with
\[
\mu=\frac{1}{4}(\varphi_{1,1}\mu+\varphi_{1,2}\mu+\varphi_{2,1}\mu+\varphi_{2,2}\mu),
\]
then
\[
\dim_{H}\mu=\frac{\log4}{\log|\theta|}>1\:.
\]

Denote by $\left\langle \cdot,\cdot\right\rangle $ the Euclidean
inner product on $\mathbb{C}$, i.e. $\left\langle z,w\right\rangle =\mathrm{Re}(z\cdot\overline{w})$
for $z,w\in\mathbb{C}$. Write
\[
S=\{z\in\mathbb{C}\::\:|z|=1\},
\]
and $P_{z}w=\left\langle w,z\right\rangle $ for $z\in S$ and $w\in\mathbb{C}$.
Note that $P_{z}\mu$ is, up to affine equivalence, the pushforward
of $\mu$ by the orthogonal projection onto the line $z\cdot\mathbb{R}$.
The following proposition is proven in Section \ref{S3}.
\begin{prop}
\label{P2.4}There exists a dense $G_{\delta}$ subset $B$ of $S$,
such that $P_{z}\mu$ is singular for all $z\in B$.
\end{prop}

Let $\mathcal{B}$ be the Borel $\sigma$-algebra of $\mathbb{C}$.
For $z\in S$ let $\{\mu_{z,w}\}_{w\in\mathbb{C}}\subset\mathcal{M}(\mathbb{C})$
be the disintegration of $\mu$ with respect to $P_{z}^{-1}(\mathcal{B})$,
as defined in Theorem 5.14 in \cite{EW}. This means that $\mu_{z,w}$
is supported on $P_{z}^{-1}(P_{z}w)$ for $w\in\mathbb{C}$, and for
each bounded \emph{$\mathcal{B}$-}measurable $f:\mathbb{C}\rightarrow\mathbb{R}$
\[
\int f\:d\mu_{z,w}=E_{\mu}(f\mid P_{z}^{-1}(\mathcal{B}))(w)\mbox{ for \ensuremath{\mu}-a.e. \ensuremath{w\in\mathbb{C}}\:.}
\]
Here the right hand side is the conditional expectation of $f$ given
$P_{z}^{-1}(\mathcal{B})$ with respect to $\mu$. In order to prove
Theorem \ref{T1.1} it remains to establish the following proposition,
which is done in Section \ref{S4}.
\begin{prop}
\label{P2.5}There exists a dense $G_{\delta}$ subset $B$ of $S$,
such that for each $z\in B$ it holds that $\mu_{z,w}$ is discrete
for $\mu$-a.e. $w\in\mathbb{C}$.$\newline$
\end{prop}

\subsection*{Proofs of Lemmas \ref{L2} and \ref{L3}}
\begin{proof}[Proof of Lemma \ref{L2}]
By the assumptions on $\theta$ there exists $\alpha\in\mathbb{C}$,
with $|\alpha|<1$, such that $\overline{\theta}$ and $\alpha$ are
the Galois conjugates of $\theta$. Set $E=\mathbb{Q}[\theta,\overline{\theta},\alpha]$,
let $f\in\mathbb{Q}[X]$ be the minimal polynomial of $\theta$ over
$\mathbb{Q}$, and let $G$ be the Galois group of the field extension
$E/\mathbb{Q}$. Note that $E$ is a splitting field for $f$ over
$\mathbb{Q}$, and that the roots of $f$ are $\theta$, $\overline{\theta}$
and $\alpha$. It follows, by Lemma 18.3 in \cite{I}, that the action
of $G$ on $\{\theta,\overline{\theta},\alpha\}$ induces an isomorphism
from $G$ into a subgroup of $S_{3}$, where $S_{3}$ is the symmetric
group on $3$ letters. It also follows, by Theorem 18.13 in \cite{I},
that the extension $E/\mathbb{Q}$ is Galois. Hence, from Corollary
18.19 and Lemma 17.6 in \cite{I}, we get
\[
|G|=[E:\mathbb{Q}]=[E:\mathbb{Q}(\theta)]\cdot[\mathbb{Q}(\theta):\mathbb{Q}]=[E:\mathbb{Q}(\theta)]\cdot3,
\]
which shows that $3$ divides $|G|$. Let $\sigma\in G$ be with $\sigma(\beta)=\overline{\beta}$
for $\beta\in E$, then $\sigma$ has order $2$. This implies that
$2$ divides $|G|$, and so it must hold that $G$ is isomorphic to
$S_{3}$.

Now assume by contradiction that $\arg\theta\in\pi\mathbb{Q}$, then
$\theta^{n}\in\mathbb{R}$ for some $n\ge1$. Let $\tau\in G$ be
such that $\tau(\theta)=\theta$, $\tau(\overline{\theta})=\alpha$
and $\tau(\alpha)=\overline{\theta}$. Since $\tau$ and $\sigma$
are distinct, both have order $2$, and $G$ is isomorphic to $S_{3}$,
it follows that the group generated by $\tau$ and $\sigma$ is $G$.
Clearly $\tau(\theta^{n})=\theta^{n}$ and from $\theta^{n}\in\mathbb{R}$
we get $\sigma(\theta^{n})=\theta^{n}$, hence $\eta(\theta^{n})=\theta^{n}$
for all $\eta\in G$. Let $\eta\in G$ be with $\eta(\theta)=\alpha$,
then
\[
\theta^{n}=\eta(\theta^{n})=\eta(\theta)^{n}=\alpha^{n}\:.
\]
But we also have $|\theta^{n}|>1>|\alpha^{n}|$, which yields a contradiction,
and so it must holds that $\arg\theta\notin\pi\mathbb{Q}$.
\end{proof}

\begin{proof}[Proof of Lemma \ref{L3}]
Let $z\in\mathbb{C}$ and $\epsilon>0$ be given, and let $N\ge1$
be with $|\lambda^{N}|<\epsilon$. Since
\[
\arg\lambda=-\arg\theta\notin\pi\mathbb{Q},
\]
we have that
\[
\{l\cdot\arg\lambda\mod2\pi\}_{l=N}^{\infty}
\]
is dense in $[0,2\pi)$. It follows there exists $l\ge N$ with
\begin{equation}
|\exp(i\cdot\arg(\lambda^{l}))-\exp(i\cdot\arg z)|<\epsilon\:.\label{E3}
\end{equation}
Let $k\ge0$ be the integer with $k\cdot|\lambda^{l}|\le|z|<(k+1)\cdot|\lambda^{l}|$,
then
\[
\left||z|-|k\cdot\lambda^{l}|\right|\le|\lambda^{N}|<\epsilon\:.
\]
From this, from $\arg(k\cdot\lambda^{l})=\arg(\lambda^{l})$, and
from (\ref{E3}), the lemma follows.
\end{proof}

\section{\label{S3}Proof of Proposition \ref{P2.4}}

Let $\theta$, $\lambda$, $(a_{1},a_{2})$, $\Phi:=\{\varphi_{k,j}\::\:k,j\in\{1,2\}\}$
and $\mu$, as obtained in Section \ref{S2}. We shall show that there
exists a dense $G_{\delta}$ subset $B$ of $S$, such that for every
$z\in B$ the Fourier transform of $P_{z}\mu$ does not decay to $0$
at infinity.
\begin{lem}
\label{L4}There exist constants $\rho\in(0,1)$ and $C>0$, with
\begin{equation}
\mathrm{dist}(2\mathrm{Re}(\theta^{n}),\mathbb{Z})\le C\cdot\rho^{|n|}\:\mbox{ for all }n\in\mathbb{Z}\:.\label{E4}
\end{equation}

\end{lem}

\begin{proof}
Let $\theta_{3},...,\theta_{m}$ be the Galois conjugates of $\theta$
other than $\overline{\theta}$. Since $\theta$ is an algebraic integer,
\[
\theta^{n}+\overline{\theta^{n}}+\sum_{j=3}^{m}\theta_{m}^{n}\in\mathbb{Z}\:\mbox{ for all }n\in\mathbb{N}\:.
\]
From $|\theta_{j}|<1$ for $3\le j\le m$, it follows there exists
$\rho\in(0,1)$ and $C>0$ such that (\ref{E4}) holds for $n\in\mathbb{N}$.
Since $|\theta|>1$ and for each integer $n<0$
\[
\mathrm{dist}(2\mathrm{Re}(\theta^{n}),\mathbb{Z})\le2|\theta|^{n},
\]
the lemma follows.
\end{proof}

Given $\nu\in\mathcal{M}(\mathbb{C})$ let $\mathcal{F}(\nu)$ be
the Fourier transform of $\nu$ as a measure on $\mathbb{R}^{2}$,
i.e. for $\xi\in\mathbb{C}$
\[
\mathcal{F}(\nu)(\xi)=\int_{\mathbb{C}}e^{i\left\langle z,\xi\right\rangle }\:d\nu(z)=\int_{\mathbb{C}}\exp(i\mathrm{Re}(z\cdot\overline{\xi}))\:d\nu(z)\:.
\]
The proof of the following proposition resembles the argument given
by Erd\H{o}s in \cite{E}.
\begin{prop}
\label{3.2}There exists a constant $c>0$ with $|\mathcal{F}(\mu)(4\pi\overline{\theta^{N}})|>c$
for all $N\in\mathbb{N}$.
\end{prop}

\begin{proof}
Let $X_{1},X_{2},...$ be i.i.d. random variables with,
\[
\mathbb{P}(X_{1}=(-1)^{k}a_{j})=\frac{1}{4}\mbox{ for }k,j\in\{1,2\}\:.
\]
Since $\mu$ is the unique Borel probability measure on $\mathbb{C}$
with
\[
\mu=\frac{1}{4}(\varphi_{1,1}\mu+\varphi_{1,2}\mu+\varphi_{2,1}\mu+\varphi_{2,2}\mu),
\]
it is equal to the distribution of the random sum $\sum_{n=0}^{\infty}\lambda^{n}\cdot X_{n}$.
Hence for every $\xi\in\mathbb{C}$,
\begin{multline*}
\mathcal{F}(\mu)(\xi)=\prod_{n=0}^{\infty}\mathcal{F}(\frac{1}{4}\cdot\sum_{j=1}^{2}(\delta_{\lambda^{n}a_{j}}+\delta_{-\lambda^{n}a_{j}}))(\xi)\\
=\prod_{n=0}^{\infty}\frac{1}{4}\cdot\sum_{j=1}^{2}(\exp(i\mathrm{Re}(\lambda^{n}a_{j}\cdot\overline{\xi}))+\exp(i\mathrm{Re}(-\lambda^{n}a_{j}\cdot\overline{\xi})))\\
=\prod_{n=0}^{\infty}\frac{1}{2}\cdot\left(\cos(\mathrm{Re}(\lambda^{n}a_{1}\cdot\overline{\xi}))+\cos(\mathrm{Re}(\lambda^{n}a_{2}\cdot\overline{\xi}))\right)\:.
\end{multline*}
Since $a_{1},a_{2}\in\mathcal{Y}$, where $\mathcal{Y}$ is defined
in Lemma \ref{L3}, for $j=1,2$ there exist $k_{j},l_{j}\in\mathbb{N}$
with $a_{j}=k_{j}\cdot\theta^{-l_{j}}$. Hence for $N\in\mathbb{N}$,
\begin{multline}
\mathcal{F}(\mu)(4\pi\overline{\theta^{N}})=\prod_{n=-\infty}^{N}\frac{1}{2}\cdot\left(\cos(4\pi\mathrm{Re}(\theta^{n}a_{1}))+\cos(4\pi\mathrm{Re}(\theta^{n}a_{2}))\right)\\
=\prod_{n=-\infty}^{N}\frac{1}{2}\left(\cos(4\pi k_{1}\cdot\mathrm{Re}(\theta^{n-l_{1}}))+\cos(4\pi k_{2}\cdot\mathrm{Re}(\theta^{n-l_{2}}))\right)\:.\label{E100}
\end{multline}

Let us show that $b_{n}\ne0$ for every $n\in\mathbb{Z}$, where
\[
b_{n}:=\frac{1}{2}\left(\cos(4\pi k_{1}\cdot\mathrm{Re}(\theta^{n-l_{1}}))+\cos(4\pi k_{2}\cdot\mathrm{Re}(\theta^{n-l_{2}}))\right)\:.
\]
Recall that the set of algebraic integers is closed under addition,
subtraction and multiplication. The product of $\theta$ with its
Galois conjugates is equal to the constant term of the minimal polynomial
of $\theta$, which is $\pm1$ by assumption. These conjugates are
all algebraic integers, hence $\theta^{-1}$ is an algebraic integer,
and so $\theta^{n}$ is an algebraic integer for all $n\in\mathbb{Z}$.
Let $n\in\mathbb{Z}$, then from the identity
\[
\cos\beta+\cos\gamma=2\cos(\frac{\beta+\gamma}{2})\cos(\frac{\beta-\gamma}{2})\mbox{ for all }\beta,\gamma\in\mathbb{R},
\]
we obtain
\begin{equation}
b_{n}=\cos\left(2\pi\cdot\mathrm{Re}(k_{1}\theta^{n-l_{1}}+k_{2}\theta^{n-l_{2}})\right)\cdot\cos\left(2\pi\cdot\mathrm{Re}(k_{1}\theta^{n-l_{1}}-k_{2}\theta^{n-l_{2}})\right)\:.\label{E1}
\end{equation}
Since $2\mathrm{Re}(k_{1}\theta^{n-l_{1}}+k_{2}\theta^{n-l_{2}})$
is equal to
\[
k_{1}\theta^{n-l_{1}}+k_{2}\theta^{n-l_{2}}+k_{1}\overline{\theta^{n-l_{1}}}+k_{2}\overline{\theta^{n-l_{2}}},
\]
it is an algebraic integer, and so it can't be of the form $k+\frac{1}{2}$
with $k\in\mathbb{Z}$. It follows the first term in the product (\ref{E1})
is nonzero. In a similar manner the second term in (\ref{E1}) is
nonzero, which shows $b_{n}\ne0$.

Fix $n\in\mathbb{Z}$ and $j\in\{1,2\}$, and let $d\in\mathbb{Z}$
be with 
\[
|2\mathrm{Re}(\theta^{n-l_{j}})-d|=\mathrm{dist}(2\mathrm{Re}(\theta^{n-l_{j}}),\mathbb{Z})\:.
\]
Let $C$ and $\rho$ be the constants from Lemma \ref{L4}, and write
\[
C_{0}:=2\pi C\cdot\max\{k_{1},k_{2}\}\cdot\rho^{-\max\{l_{1},l_{2}\}}\:.
\]
From Lemma \ref{L4}, 
\begin{multline*}
|\cos(4\pi k_{j}\cdot\mathrm{Re}(\theta^{n-l_{j}}))-1|=|\cos(4\pi k_{j}\cdot\mathrm{Re}(\theta^{n-l_{j}}))-\cos(2\pi k_{j}d)|\\
\le2\pi k_{j}\cdot|2\mathrm{Re}(\theta^{n-l_{j}})-d|=2\pi k_{j}\cdot\mathrm{dist}(2\mathrm{Re}(\theta^{n-l_{j}}),\mathbb{Z})\\
\le2\pi k_{j}C\cdot\rho^{|n-l_{j}|}\le C_{0}\cdot\rho^{|n|}\:.
\end{multline*}
This shows,
\[
|b_{n}|\ge1-\frac{1}{2}\sum_{j=1}^{2}|\cos(4\pi k_{j}\cdot\mathrm{Re}(\theta^{n-l_{j}}))-1|\ge1-C_{0}\cdot\rho^{|n|}\:.
\]
Now let $M\ge1$ be such that $C_{0}\cdot\rho^{|n|}<1$ for all $n\in\mathbb{Z}$
with $|n|\ge M$. Then from (\ref{E100}) it follows that for each
$N\ge0$,
\begin{multline*}
|\mathcal{F}(\mu)(4\pi\overline{\theta^{N}})|\ge\prod_{n=-\infty}^{-M}|b_{n}|\prod_{n=1-M}^{M-1}|b_{n}|\prod_{n=M}^{\infty}|b_{n}|\\
\ge\prod_{n=M}^{\infty}(1-C_{0}\cdot\rho^{n})^{2}\cdot\prod_{n=1-M}^{M-1}|b_{n}|>0,
\end{multline*}
which completes the proof.
\end{proof}
Let $\mathcal{M}(\mathbb{R})$ be the collection of all compactly
supported Borel probability measures on $\mathbb{R}$. Given $\nu\in\mathcal{M}(\mathbb{R})$
let $\mathcal{F}(\nu)$ be the Fourier transform of $\nu$, i.e.
\[
\mathcal{F}(\nu)(r)=\int_{\mathbb{R}}e^{ixr}\:d\nu(x)\mbox{ for \ensuremath{r\in\mathbb{R}}}.
\]
Recall that
\[
S=\{z\in\mathbb{C}\::\:|z|=1\},
\]
and $P_{z}w=\left\langle w,z\right\rangle $ for $z\in S$ and $w\in\mathbb{C}$.
For $n\ge1$ write
\[
U_{n}=\{z\in S\::\:\underset{r\ge n}{\sup}\:|\mathcal{F}(P_{z}\mu)(r)|>c\},
\]
where $c$ is the constant from Proposition \ref{3.2}.
\begin{lem}
\label{L6}Let $n\ge1$, then $U_{n}$ is an open and dense subset
of $S$.
\end{lem}

\begin{proof}
Note that for $z\in S$ and $r\in\mathbb{R}$
\[
\mathcal{F}(P_{z}\mu)(r)=\int_{\mathbb{R}}\exp(ixr)\:dP_{z}\mu(x)=\int_{\mathbb{R}}\exp(i\left\langle w,rz\right\rangle )\:d\mu(w)=\mathcal{F}(\mu)(rz),
\]
hence
\[
U_{n}=\{z\in S\::\:\underset{r\ge n}{\sup}\:|\mathcal{F}(\mu)(rz)|>c\}\:.
\]
Now since $\mathcal{F}(\mu)$ is continuous it follows $U_{n}$ is
open in $S$. Set $\eta=\exp(-i\arg\theta)$, then from Proposition
\ref{3.2}
\begin{equation}
|\mathcal{F}(P_{\eta^{k}}\mu)(4\pi|\theta|^{k})|=|\mathcal{F}(\mu)(4\pi\overline{\theta^{k}})|>c\label{E2}
\end{equation}
for every integer $k\ge0$. Let $N\ge1$ be with $|4\pi\theta^{N}|\ge n$,
then $\{\eta^{k}\}_{k=N}^{\infty}\subset U_{n}$ by (\ref{E2}). By
assumption $\arg\theta\notin\pi\mathbb{Q}$, hence $\{\eta^{k}\}_{k=N}^{\infty}$
is dense in $S$, which proves the lemma.
\end{proof}
We can now complete the proof of Proposition \ref{P2.4}.
\begin{proof}[Proof of Proposition \ref{P2.4}]
Set $B=\cap_{n=1}^{\infty}U_{n}$, then $B$ is a dense $G_{\delta}$
subset of $S$ by Lemma \ref{L6} and Baire's theorem. Let $z\in B$,
then $\mathcal{F}(P_{z}\mu)(r)$ does not tend to $0$ as $r\rightarrow\infty$.
Hence, by the Riemann-Lebesgue lemma, $P_{z}\mu$ is not absolutely
continuous. From the law of pure types (see Theorem 3.26 in \cite{B})
it now follows $P_{z}\mu$ is singular, which completes the proof
of the Proposition.
\end{proof}

\section{\label{S4}Proof of Proposition \ref{P2.5}}

In order to prove Proposition \ref{P2.5} we shall use the following
theorem due to Wiener (see Section VI.2.12 of \cite{K}).
\begin{thm}
\label{T4.1}For every $\nu\in\mathcal{M}(\mathbb{R})$,
\[
\sum_{x\in\mathbb{R}}\left(\nu\{x\}\right)^{2}=\underset{M\rightarrow\infty}{\lim}\:\frac{1}{2M}\int_{-M}^{M}|\mathcal{F}(\nu)(\xi)|^{2}\:d\xi\:.
\]

\end{thm}

Let $\theta$, $\Phi$ and $\mu$ be as in Section \ref{S2}. For
$z\in S$ write $z^{\perp}=e^{-i\pi/2}z$. Given $n,k\in\mathbb{N}$
set
\[
J_{n,k}=\{z\in S\::\:\left\langle 4\pi\overline{\theta^{k}},z^{\perp}\right\rangle \in(n,n+1)\},
\]
and let $V_{n}=\cup_{k\in\mathbb{N}}J_{n,k}$.
\begin{lem}
\label{L4.2}Let $n\in\mathbb{N}$, then $V_{n}$ is a dense open
subset of $S$.
\end{lem}

\begin{proof}
Since $J_{n,k}$ is open in $S$ for every $k\in\mathbb{N}$ the same
holds for $V_{n}$. Let $a\in\mathbb{R}$ and $0<\epsilon<1$ be given.
For $E\subset\mathbb{R}$ write $q(E)=E+2\pi\mathbb{Z}$. Since $|\theta|>1$
and $\arg\theta\notin\pi\mathbb{Q}$, there exists $k\in\mathbb{N}$
with
\[
|4\pi\overline{\theta^{k}}|>\frac{n+1}{\cos(\frac{\pi}{2}-\epsilon)}\:\mbox{ and }\:\arg(4\pi\overline{\theta^{k}})\in q(a-\epsilon,a+\epsilon)\:.
\]
Set $w=4\pi\overline{\theta^{k}}$ and for $t\in\mathbb{R}$ write
$f(t)=\left\langle w,e^{it}\right\rangle $. It holds that
\[
f(\arg w^{\perp})=\left\langle w,\frac{w^{\perp}}{|w^{\perp}|}\right\rangle =0,
\]
\begin{multline*}
f(\arg w^{\perp}+\epsilon)=f(\arg w-\frac{\pi}{2}+\epsilon)=\left\langle w,\frac{w}{|w|}\cdot e^{i(\epsilon-\pi/2)}\right\rangle \\
=\mathrm{Re}(w\cdot\frac{\overline{w}}{|w|}\cdot e^{i(\pi/2-\epsilon)})=|w|\cdot\cos(\frac{\pi}{2}-\epsilon)>n+1,
\end{multline*}
and
\[
[\arg w^{\perp},\arg w^{\perp}+\epsilon]\subset q[a-\epsilon-\frac{\pi}{2},a+2\epsilon-\frac{\pi}{2}]\:.
\]
Hence, since $f$ is continuous and $2\pi$-periodic, there exists
$t\in[a-\epsilon-\frac{\pi}{2},a+2\epsilon-\frac{\pi}{2}]$ with $f(t)\in(n,n+1)$.
Set $z=\exp(i(t+\frac{\pi}{2}))$, then
\[
\left\langle w,z^{\perp}\right\rangle =\left\langle w,e^{it}\right\rangle =f(t)\in(n,n+1),
\]
and so $z\in J_{n,k}\subset V_{n}$. Now since $a$ and $\epsilon$
are arbitrary and
\[
\arg z\in q\{t+\frac{\pi}{2}\}\subset q[a-\epsilon,a+2\epsilon],
\]
it follows that $V_{n}$ is dense in $S$, which proves the lemma.
\end{proof}

Set $B=\cap_{n\in\mathbb{N}}V_{n}$, then $B$ is a dense $G_{\delta}$
subset of $S$ by Lemma \ref{L4.2}. Fix $z\in B$ and recall that
$\{\mu_{z,w}\}_{w\in\mathbb{C}}$ is the disintegration of $\mu$
with respect to $P_{z}^{-1}(\mathcal{B})$, where $\mathcal{B}$ is
the Borel $\sigma$-algebra of $\mathbb{C}$. In order to prove the
proposition, it suffices to show that $\mu_{z,w}$ is discrete for
$\mu$-a.e. $w\in\mathbb{C}$. Write
\begin{equation}
\tau_{w}(\xi)=\xi-w\:\mbox{ and }\:R\xi=\overline{z^{\perp}}\cdot\xi\:\mbox{ for }w,\xi\in\mathbb{C},\label{E17}
\end{equation}
for each $w\in\mathbb{C}$ let $\nu_{w}=R\tau_{w}\mu_{z,w}$, and
note that $\nu_{w}\in\mathcal{M}(\mathbb{R})$.
\begin{lem}
\label{L4.3}Let $c>0$ be the constant from Proposition \ref{3.2},
then for each $n\in\mathbb{N}$ there exists $t_{n}\in(n,n+1)$ with
\begin{equation}
\int|\mathcal{F}(\nu_{w})(t_{n})|^{2}\:d\mu(w)>c^{2}\:.\label{E10}
\end{equation}

\end{lem}

\begin{proof}
Let $n\in\mathbb{N}$. Since $z\in V_{n}$ there exists $k_{n}\in\mathbb{N}$
and $t_{n}\in(n,n+1)$ with $\left\langle 4\pi\overline{\theta^{k_{n}}},z^{\perp}\right\rangle =t_{n}$.
Write $\eta=4\pi\overline{\theta^{k_{n}}}$, then by Proposition \ref{3.2},
\begin{equation}
c<\left|\mathcal{F}(\mu)(\eta)\right|\le\int\left|\int e^{i\left\langle \xi,\eta\right\rangle }\:d\mu_{z,w}(\xi)\right|\:d\mu(w)\:.\label{E16}
\end{equation}
Let $Q_{z^{\perp}}$ be the orthogonal projection onto $z^{\perp}\mathbb{R}$,
i.e.
\[
Q_{z^{\perp}}\xi=\left\langle \xi,z^{\perp}\right\rangle z^{\perp}\:\mbox{ for \ensuremath{\xi\in\mathbb{C}}\:.}
\]
From (\ref{E16}) and since $\tau_{w}\mu_{z,w}$ is supported on $z^{\perp}\mathbb{R}$
for $w\in\mathbb{C}$,
\begin{multline*}
c<\int\left|\int e^{i\left\langle \xi+w,\eta\right\rangle }\:d\tau_{w}\mu_{z,w}(\xi)\right|\:d\mu(w)\\
=\int\left|e^{i\left\langle w,\eta\right\rangle }\right|\cdot\left|\int e^{i\left\langle \xi,\eta\right\rangle }\:dQ_{z^{\perp}}\tau_{w}\mu_{z,w}(\xi)\right|\:d\mu(w)\\
=\int\left|\int\exp(i\left\langle Q_{z^{\perp}}\xi,\eta\right\rangle )\:d\tau_{w}\mu_{z,w}(\xi)\right|\:d\mu(w)\:.
\end{multline*}
Now since $Q_{z^{\perp}}$ is self-adjoint, $\left\langle \eta,z^{\perp}\right\rangle $
is equal to $t_{n}$, and $R$ from (\ref{E17}) is a rotation,
\begin{multline*}
c<\int\left|\int\exp(i\left\langle \xi,t_{n}z^{\perp}\right\rangle )\:d\tau_{w}\mu_{z,w}(\xi)\right|\:d\mu(w)\\
=\int\left|\int\exp(i\left\langle R\xi,t_{n}\cdot Rz^{\perp}\right\rangle )\:d\tau_{w}\mu_{z,w}(\xi)\right|\:d\mu(w)\\
=\int\left|\int e^{i\xi t_{n}}\:d\nu_{w}(\xi)\right|\:d\mu(w)=\int\left|\mathcal{F}(\nu_{w})(t_{n})\right|\:d\mu(w)\:.
\end{multline*}
From this and Jensen's inequality the lemma follows.
\end{proof}

Let us define the set,
\begin{equation}
E_{z}=\{w\in\mathbb{C}\::\:\mu_{z,w}\{w\}>0\}\:.\label{E11}
\end{equation}

\begin{lem}
\label{L4.4}It holds that $\mu(E_{z})>0$.
\end{lem}

\begin{proof}
Let $\{t_{n}\}_{n\in\mathbb{N}}$ be the numbers obtained in Lemma
\ref{L4.3}. Since $\mathrm{supp}(\mu)$ is compact and
\[
\mathrm{supp}(\mu_{z,w})\subset\mathrm{supp}(\mu)\mbox{ for \ensuremath{\mu}-a.e. \ensuremath{w\in\mathbb{C}},}
\]
there exists $M>0$ such that $\nu_{w}$ is supported on $[-M,M]$
for $\mu$-a.e. $w\in\mathbb{C}$. It follows that $\mathcal{F}(\nu_{w})$
is $M$-Lipschitz for $\mu$-a.e. $w\in\mathbb{C}$. Hence there exist
$\delta>0$ and intervals $\{A_{n}\}_{n\in\mathbb{N}}$, such that
for every $n\in\mathbb{N}$ it holds
\[
t_{n}\in A_{n}\subset(n,n+1),
\]
$A_{n}$ has length $\delta$, and for $\mu$-a.e. $w\in\mathbb{C}$,
\[
\left|\left|\mathcal{F}(\nu_{w})(t_{n})\right|^{2}-\left|\mathcal{F}(\nu_{w})(x)\right|^{2}\right|<\frac{c^{2}}{2}\mbox{ for }x\in A_{n}\:.
\]
We now get from (\ref{E10}) that for each $N\ge1$,
\begin{multline*}
c^{2}\le\int\frac{1}{N}\sum_{n=0}^{N-1}|\mathcal{F}(\nu_{w})(t_{n})|^{2}\:d\mu(w)\\
=\int\frac{1}{\delta N}\sum_{n=0}^{N-1}\int_{A_{n}}|\mathcal{F}(\nu_{w})(t_{n})|^{2}\:dx\:d\mu(w)\\
\le\int\frac{1}{\delta N}\sum_{n=0}^{N-1}\int_{A_{n}}|\mathcal{F}(\nu_{w})(x)|^{2}+\frac{c^{2}}{2}\:dx\:d\mu(w)\\
\le\int\frac{1}{\delta N}\int_{-N}^{N}|\mathcal{F}(\nu_{w})(x)|^{2}\:dx\:d\mu(w)+\frac{c^{2}}{2},
\end{multline*}
which gives,
\[
\frac{\delta c^{2}}{4}\le\int\frac{1}{2N}\int_{-N}^{N}|\mathcal{F}(\nu_{w})(x)|^{2}\:dx\:d\mu(w)\:.
\]
Now by Theorem \ref{T4.1} and the bounded convergence theorem,
\begin{multline*}
\int\sum_{\xi\in\mathbb{C}}\mu_{z,w}\{\xi\}\:d\mu(w)=\int\sum_{x\in\mathbb{R}}\nu_{w}\{x\}\:d\mu(w)\\
=\int\underset{N\rightarrow\infty}{\lim}\:\frac{1}{2N}\int_{-N}^{N}|\mathcal{F}(\nu_{w})(\xi)|^{2}\:d\xi\:d\mu(w)\\
=\underset{N\rightarrow\infty}{\lim}\:\int\frac{1}{2N}\int_{-N}^{N}|\mathcal{F}(\nu_{w})(\xi)|^{2}\:d\xi\:d\mu(w)\ge\frac{\delta c^{2}}{4}>0\:.
\end{multline*}
This gives $\mu(F_{z})>0$, where
\[
F_{z}=\{w\in\mathbb{C}\::\:\mu_{z,w}\{\xi\}>0\mbox{ for some }\xi\in\mathbb{C}\}\:.
\]
Let $w\in F_{z}$, then there exists
\[
\xi\in\mathrm{supp}(\mu_{z,w})\subset w+z^{\perp}\mathbb{R}
\]
with $\mu_{z,w}\{\xi\}>0$. Since $\mu_{z,\xi}=\mu_{z,w}$ it follows
$\xi\in E_{z}$, where $E_{z}$ is defined in (\ref{E11}), and so
\[
\mu_{z,w}(E_{z})\ge\mu_{z,w}\{\xi\}>0\:.
\]
Now from $\mu(F_{z})>0$ we get
\[
\mu(E_{z})\ge\int_{F_{z}}\mu_{z,w}(E_{z})\:d\mu(w)>0,
\]
which proves the lemma.
\end{proof}

Write $I=\{1,2\}^{2}$ and let $\Phi=\{\varphi_{i}\}_{i\in I}$ be
the IFS constructed in Section \ref{S2}. Recall that $\lambda=\theta^{-1}$,
for each $i\in I$ there exists $a_{i}\in\mathbb{C}$ with $\varphi(w)=\lambda w+a_{i}$
for $w\in\mathbb{C}$, and $\Phi$ satisfies the SSC. Let $K\subset\mathbb{C}$
be attractor of $\Phi$, write $\mathcal{B}_{K}$ for the restriction
of the Borel $\sigma$-algebra $\mathcal{B}$ to $K$, and let $T:K\rightarrow K$
be such that $Tx=\varphi_{i}^{-1}(x)$ for $i\in I$ and $x\in\varphi_{i}(K)$.
\begin{lem}
\label{L4.5}It holds that $\mu(E_{z})=1$.
\end{lem}

\begin{proof}
The system $(K,\mathcal{B}_{K},T,\mu)$ is measure preserving and
isomorphic to a Bernoulli shift. We shall show that
\[
E_{z}\in\cap_{n=0}^{\infty}T^{-n}(\mathcal{B}_{K})\:\mod\mu,
\]
from which the lemma will follow by the zero-one law.

Given a word $i_{1}\cdot...\cdot i_{n}=\alpha\in I^{*}$ write $\varphi_{\alpha}=\varphi_{i_{1}}\circ...\circ\varphi_{i_{n}}$
and $K_{\alpha}=\varphi_{\alpha}(K)$. For $n\in\mathbb{N}$ and $w\in K$
let $\alpha_{n}(w)\in I^{n}$ be the unique word of length $n$ for
which $w\in K_{\alpha_{n}(w)}$, where $\alpha_{0}(w)$ is the empty
word $\emptyset$ and $K_{\emptyset}=K$. For $m,n\in\mathbb{N}$
and $w\in K$ set
\[
z_{m}=\frac{\theta^{m}z}{|\theta^{m}z|}\in S\:\mbox{ and }\:F_{m,n}(w)=\mu_{z_{m},w}(K_{\alpha_{n}(w)})\:.
\]
For $m\in\mathbb{N}$, $w\in\mathbb{C}$ and $\delta>0$ let
\[
V_{w}^{m}(\delta)=w+z_{m}^{\perp}\cdot\mathbb{R}+B(0,\delta),
\]
where $B(0,\delta)$ is the open disk in $\mathbb{C}$ with centre
$0$ and radius $\delta$. From Lemma 3.3 in \cite{FH} we get that
for each $m\in\mathbb{N}$ and $A\in\mathcal{B}$,
\[
\mu_{z_{m},w}(A)=\underset{\delta\downarrow0}{\lim}\:\frac{\mu(V_{w}^{m}(\delta)\cap A)}{\mu(V_{w}^{m}(\delta))}\mbox{ for \ensuremath{\mu}-a.e. \ensuremath{w\in\mathbb{C}}.}
\]

Fix $m,n\in\mathbb{N}$, then for $\mu$-a.e. $w\in K$
\begin{equation}
F_{m,n}(T^{m}w)=\frac{\mu_{z_{m},T^{m}w}(K_{\alpha_{n}(T^{m}w)})}{\mu_{z_{m},T^{m}w}(K_{\alpha_{0}(T^{m}w)})}=\underset{\delta\downarrow0}{\lim}\:\frac{\mu(V_{T^{m}w}^{m}(\delta)\cap K_{\alpha_{n}(T^{m}w)})}{\mu(V_{T^{m}w}^{m}(\delta)\cap K_{\alpha_{0}(T^{m}w)})}\:.\label{E18}
\end{equation}
Since $\mu$ satisfies the SSC,
\[
\varphi_{\alpha}(K)\cap\varphi_{\alpha_{m}(w)}(K)=\emptyset\:\mbox{ for }\alpha\in I^{m}\setminus\{\alpha_{m}(w)\},
\]
hence,
\[
\mu(\varphi_{\alpha}^{-1}(\varphi_{\alpha_{m}(w)}(K))=0\:\mbox{ for }\alpha\in I^{m}\setminus\{\alpha_{m}(w)\}\:.
\]
From this and (\ref{E18}) it follows that for $\mu$-a.e. $w\in K$,
\[
F_{m,n}(T^{m}w)=\underset{\delta\downarrow0}{\lim}\:\frac{\sum_{\alpha\in I^{m}}|I|^{-m}\cdot\mu(\varphi_{\alpha}^{-1}(\varphi_{\alpha_{m}(w)}(V_{T^{m}w}^{m}(\delta)\cap K_{\alpha_{n}(T^{m}w)})))}{\sum_{\alpha\in I^{m}}|I|^{-m}\cdot\mu(\varphi_{\alpha}^{-1}(\varphi_{\alpha_{m}(w)}(V_{T^{m}w}^{m}(\delta)\cap K_{\alpha_{0}(T^{m}w)})))}\:.
\]
Now since $\mu=\sum_{\alpha\in I^{m}}|I|^{-m}\cdot\varphi_{\alpha}\mu$,
\begin{equation}
F_{m,n}(T^{m}w)=\underset{\delta\downarrow0}{\lim}\:\frac{\mu(\varphi_{\alpha_{m}(w)}(V_{T^{m}w}^{m}(\delta)\cap K_{\alpha_{n}(T^{m}w)}))}{\mu(\varphi_{\alpha_{m}(w)}(V_{T^{m}w}^{m}(\delta)\cap K_{\alpha_{0}(T^{m}w)}))}\:\mbox{ for \ensuremath{\mu}-a.e. \ensuremath{w\in K}\:.}\label{E12}
\end{equation}
For every $\delta>0$,
\begin{multline*}
\varphi_{\alpha_{m}(w)}(V_{T^{m}w}^{m}(\delta))=\varphi_{\alpha_{m}(w)}(T^{m}w)+\lambda^{m}z_{m}^{\perp}\cdot\mathbb{R}+\lambda^{m}\cdot B(0,\delta)\\
=w+z_{0}^{\perp}\cdot\mathbb{R}+B(0,|\delta\lambda^{m}|)=V_{w}^{0}(|\delta\lambda^{m}|)\:.
\end{multline*}
Hence from (\ref{E12}) it follows that for $\mu$-a.e. $w\in K$,
\begin{equation}
F_{m,n}(T^{m}w)=\underset{\delta\downarrow0}{\lim}\:\frac{\mu(V_{w}^{0}(|\delta\lambda^{m}|)\cap K_{\alpha_{m+n}(w)})}{\mu(V_{w}^{0}(|\delta\lambda^{m}|)\cap K_{\alpha_{m}(w)}))}=\frac{\mu_{z,w}(K_{\alpha_{m+n}(w)})}{\mu_{z,w}(K_{\alpha_{m}(w)})},\label{E13}
\end{equation}
where we have used the fact that $\mu_{z,w}(K_{\alpha_{m}(w)})>0$
for $\mu$-a.e. $w\in K$.

Let $m\in\mathbb{N}$, then from (\ref{E13}) we get that $\mathrm{mod}\:\mu$
it holds
\begin{multline*}
E_{z}=\{w\in K\::\:\underset{n}{\lim}\:\frac{\mu_{z,w}(K_{\alpha_{m+n}(w)})}{\mu_{z,w}(K_{\alpha_{m}(w)})}>0\}\\
=\{w\in K\::\:\underset{n}{\lim}\:F_{m,n}(T^{m}w)>0\}\in T^{-m}(\mathcal{B}_{K}),
\end{multline*}
which shows,
\[
E_{z}\in\cap_{m\in\mathbb{N}}T^{-m}(\mathcal{B}_{K})\mod\mu\:.
\]
Now since $(K,\mathcal{B}_{K},T,\mu)$ is isomorphic to a Bernoulli
shift, it follows that $\mu(E_{z})=0\mbox{ or }1$. But by Lemma \ref{L4.4}
we have $\mu(E_{z})>0$, which completes the proof of the lemma.
\end{proof}

We can now complete the proof of Proposition \ref{P2.5}.
\begin{proof}[Proof of Proposition \ref{P2.5}]
As mentioned above, it suffices to show $\mu_{z,w}$ is discrete
for $\mu$-a.e. $w\in\mathbb{C}$. By lemma (\ref{L4.5}) we have
$\mu(E_{z})=1$, and so $\mu_{z,w}(E_{z})=1$ for $\mu$-a.e. $w\in\mathbb{C}$.
Fix such a $w\in\mathbb{C}$ and let $A=w+z^{\perp}\mathbb{R}$. Since
$\mu_{z,w}(A)=1$ and $\mu_{z,w}=\mu_{z,\xi}$ for $\xi\in A$,
\[
1=\mu_{z,w}(E_{z}\cap A)=\mu_{z,w}\{\xi\in A\::\:\mu_{z,w}\{\xi\}>0\}\:.
\]
This shows that $\mu_{z,w}$ is discrete, which completes the proof.
\end{proof}


\begin{thebibliography}{BFVZ}
\bibitem[BFVZ]{BFVZ}V.Beresnevich, K.Falconer, S.Velani, and A.Zafeiropoulos,
Marstrand's Theorem Revisited: Projecting Sets of Dimension Zero.
Preprint available at https://arxiv.org/abs/1703.08554.

\bibitem[B]{B}L.Breiman, Probability. Addison-Wesley, Reading, Mass.,
1968.

\bibitem[EW]{EW}M.Einsiedler and T.Ward, Ergodic Theory with a View
Towards Number Theory (Graduate Texts in Mathematics, 259). Springer,
London, 2011.

\bibitem[E]{E}P.Erd\H{o}s, On a family of symmetric Bernoulli convolutions.
Amer. J. Math. 61 (1939), 974\textendash 976.

\bibitem[FJ1]{FJ1}K.Falconer and X.Jin, Exact dimensionality and
projections of random self-similar measures and sets. J. Lond. Math.
Soc. (2), 90(2):388\textendash 412, 2014.

\bibitem[FJ2]{FJ2}K.Falconer and X.Jin, Dimension conservation for
self-similar sets and fractal percolation. Int. Math. Res. Not. 2015:
13260\textendash 13289, 2015.

\bibitem[FH]{FH}D.-J. Feng and H. Hu, Dimension theory of iterated
function systems. Comm. Pure Appl. Math. 62 (2009), 1435-1500.

\bibitem[F]{F}H.Furstenberg, Ergodic fractal measures and dimension
conservation. Ergodic Theory Dynam. Systems 28 (2008), 405-422.

\bibitem[HS]{HS}M.Hochman and P.Shmerkin, Local entropy averages
and projections of fractal measures. Ann. of Math. (2), 175:1001-1059,
2012.

\bibitem[HK]{HK}B.Hunt and V.Kaloshin, How projections affect the
dimension spectrum of fractal measures. Nonlinearity 10 (1997), 1031-1046.

\bibitem[I]{I}I.M.Isaacs, Algebra: a graduate course, Graduate Studies
in Mathematics, vol. 100. Amer. Math. Soc., Prov., RI, 2009.

\bibitem[JM]{JM}M.Jarvenpaa and P.Mattila, Hausdorff and packing
dimensions and sections of measures. Mathematika 45, 1998, 55-77.

\bibitem[K]{K}Y.Katznelson, An Introduction to Harmonic Analysis.
Cambridge University Press, 2004.

\bibitem[M]{M}P.Mattila, Geometry of Sets and Measures in Euclidean
Spaces. Cambridge University Press, 1995.

\bibitem[NPS]{NPS}F.Nazarov, Y.Peres, and P.Shmerkin, Convolutions
of cantor measures without resonance. Israel J. of Math, 187(1):93-116,
2012.

\bibitem[PS]{PS}Y.Peres and P.Shmerkin, Resonance between cantor
sets. Ergodic Theory and Dynamical Systems 29 (2009), 201-221.

\bibitem[R]{R}A.Rapaport, On the Hausdorff and packing measures of
slices of dynamically defined sets. J. Frac. Geom. 3.1 (2016): 33-74.

\bibitem[SS]{SS}P.Shmerkin and B.Solomyak, Absolute continuity of
self-similar measures, their projections and convolutions. Trans.
Amer. Math. Soc., 368(7):5125\textendash 5151, 2016.

\bibitem[SV]{SV}K.Simon and L.Vágó, Singularity versus exact overlaps
for self-similar measures. Preprint available at https://arxiv.org/abs/1702.06785.

\bibitem[SX]{SX}B.Solomyak and H.Xu, On the \textquoteleft Mandelbrot
set\textquoteright{} for a pair of linear maps and complex Bernoulli
convolutions. Nonlinearity, 16(5):1733\textendash 1749, 2003.\end{thebibliography}
\end{document}